\documentclass[11pt]{article}
\usepackage{amssymb,amsmath,amsthm}
\usepackage{multirow}
\usepackage{enumerate}
\usepackage[dvipsnames]{xcolor}
\usepackage{pgf,tikz}
\usepackage{hyperref}
\usetikzlibrary{arrows}

\newenvironment{packedItem}{
\begin{itemize}
  \setlength{\itemsep}{1pt}
  \setlength{\parskip}{0pt}
  \setlength{\parsep}{0pt}
}{\end{itemize}}

\newenvironment{packedEnum}{
\begin{enumerate}
  \setlength{\itemsep}{1pt}
  \setlength{\parskip}{0pt}
  \setlength{\parsep}{0pt}
}{\end{enumerate}}

\let\oldmarginpar\marginpar
\renewcommand\marginpar[1]{\-\oldmarginpar[\raggedleft\footnotesize #1]%
{\raggedright\footnotesize #1}}

\newtheorem{theorem}{Theorem}
\newtheorem{corollary}[theorem]{Corollary}
\newtheorem{lemma}[theorem]{Lemma}
\newtheorem{proposition}[theorem]{Proposition}
\newtheorem{obs}[theorem]{Observation}

\theoremstyle{definition}

\newtheorem{definition}[theorem]{Definition}

\newtheorem{construction}{Construction}
\numberwithin{equation}{section}

\newcommand{\h}{\mathcal{H}}
\newcommand{\ucal}{\mathcal{U}}
\newcommand{\wcal}{\mathcal{W}}

\newcommand{\calk}{\mathcal{K}}
\newcommand{\kcal}{\mathcal{K}}
\newcommand{\s}{\mathcal{S}}

\newcommand{\diff}{{\rm{diff}}}

\newcommand{\nbd}{N}

\newcommand{\Nnn}{\mathbb{N}}

\newcommand{\vanish}[1]{}

\begin{document}

\title{Classification of Reconfiguration Graphs of Shortest Path Graphs With No Induced $4$-cycles\\
}

\author{
John Asplund\\
{\small Department of Technology and Mathematics,} \\
{\small Dalton State College,} \\
{\small Dalton, GA 30720, USA} \\
{\small jasplund@daltonstate.edu}\\
\\
Brett Werner\\
{\small Department of Mathematics,} \\
{\small University of Colorado, Boulder,} \\
{\small Boulder, CO, 80309, USA} \\
{\small brett.werner@colorado.edu}\\
\\
 }

\date{\today}
\maketitle

\begin{abstract}
For any graph $G$ with $a,b\in V(G)$, a shortest path reconfiguration graph can be formed with respect to $a$ and $b$; we denote such a graph as $S(G,a,b)$. The vertex set of $S(G,a,b)$ is the set of all shortest paths from $a$ to $b$ in $G$ while two vertices $U,W$ in $V(S(G,a,b))$ are adjacent if and only if the vertex sets of the paths that represent $U$ and $W$ differ in exactly one vertex. In a recent paper [Asplund et al., \textit{Reconfiguration graphs of shortest paths}, Discrete Mathematics \textbf{341} (2018), no. 10, 2938--2948], it was shown that shortest path graphs with girth five or greater are exactly disjoint unions of even cycles and paths. In this paper, we extend this result by classifying all shortest path graphs with no induced $4$-cycles.
\end{abstract}

\section{Introduction}

In reconfiguration problems, the objective is to determine whether it is possible to transform one feasible solution into a target feasible solution in a step-by-step manner (a reconfiguration), such that each intermediate solution is also feasible.
Such transformations can be studied via the reconfiguration graph, in which the vertices represent feasible solutions and there is an edge between two vertices when it is possible to get from one feasible solution to another in a single application of the reconfiguration rule.
Many types of reconfiguration problems have been studied with drastically different reconfiguration rules: 
vertex coloring \cite{BJLPP,BC,CJV,CJV2,CJV3}, independent sets \cite{HD,IDHPSUU,KMM}, matchings \cite{IDHPSUU}, list-colorings \cite{IKD}, matroid bases \cite{IDHPSUU}, and subsets of a (multi)set of numbers \cite{EW}.
This paper focuses on the reconfiguration of shortest paths in a graph. The \textit{shortest path graph} (SPG) of a graph $G$ with respect to $a,b\in V(G)$ is a graph where each vertex corresponds to a shortest path in $G$ from $a$ to $b$, or an $\mathit{(a,b)}$\textit{-geodesic}, and two vertices in the SPG are adjacent if and only their corresponding $(a,b)$-geodesics in $G$ differ at exactly one vertex.

Kami\'{n}ski, Medvedev, and Milani\v{c} \cite{KMM} showed that a family of graphs whose size is linear in $k$ has diameter of the reconfiguration graph that is $\Omega(2^k)$. That is, as the size of a graph increases, the diameter of the SPG of $G$ can be exponential.  Relatedly, Bonsma \cite{B} showed that the question of determining if there is a path in a SPG between all pairs of vertices is PSPACE-complete. For these reasons, we suspect characterizing the remaining SPGs will likely be difficult, yet important. 
Any progress in the direction of characterizing these graphs is worthwhile as mentioned in \cite{KMM}. Recent studies in the area of reconfigurability have shown an emerging pattern where the most ``natural'' problems (e.g., finding a spanning tree in $G$) can be done in polynomial-time and its reconfigurability problem (e.g., finding a spanning tree in the SPG of $G$) can also be done in polynomial-time. But since it was shown in \cite{KMM} that the reconfigurability problem (finding a shortest path between any two vertices in the SPG of $G$) is NP-hard while the ``natural'' problem is in P (finding a shortest path between any two vertices in $G$), we believe investigation into when it becomes NP-hard is worthwhile. 

Asplund et al. \cite{AAEHHNW} showed that cycles are central to characterizing SPGs. We denote a cycle of length $k$ as a $k$-cycle or as $C_k=(v_1,v_2,\ldots,v_k)$. 
One of the main results of that paper was the classification of all SPGs with girth at least $5$. 
This paper also established that induced $4$-cycles are extremely prevalent in SPGs and the structure of SPGs containing an induced $4$-cycle can be rather complex. In this paper, we continue investigating the structure of SPGs leading to a classification of all SPGs that do not contain an induced $4$-cycle. 

This paper is organized as follows. In Section~\ref{notation}, the necessary notation and terminology are introduced, and several results from \cite{AAEHHNW} that are necessary to prove the results in this paper are given. In Section~\ref{prelims}, there are some preliminary results that will simplify the main results of the paper. The main result of this paper is found in Section~\ref{girth3}, where all SPGs that contain a $3$-cycle, but no induced $4$-cycles are characterized. This result along with the girth $5$ result from \cite{AAEHHNW} classifies all SPGs with no induced $4$-cycles. 
The constructions described in Section~\ref{girth3} will be stronger than is needed to prove the main theorem. In fact, these constructions can be used on a number of SPGs to build larger SPGs.

\section{Notation, Terminology, and Previous Results} \label{notation}

As discussed in the introduction, the focus of this paper is on a specific class of reconfiguration graphs: shortest path graphs. 
\begin{definition} Let $G$ be a graph with distinct vertices $a$ and $b$. The \textit{shortest path graph} (SPG) of $G$ with respect to $a$ and $b$, denoted $S(G,a,b)$, is a graph where each vertex corresponds to a shortest path in $G$ from $a$ to $b$, and two vertices $U,W \in V(S(G,a,b))$ are adjacent if and only if their corresponding paths in $G$ differ in exactly one vertex.
\end{definition}

All graphs considered in this paper are simple. A shortest path from vertex $a$ to vertex $b$ will be called a \textit{shortest $(a,b)$-path} or an \textit{$(a,b)$-geodesic}. To reduce notational confusion when discussing vertices in $G$ versus vertices in $S(G,a,b)$, we will adopt the convention of using lower-case letters to denote vertices in $G$ and upper-case letters to denote vertices in $S(G,a,b)$. Furthermore, for convenience, a vertex $U$ in $V(S(G,a,b))$ will be referred to as both a vertex in $S(G,a,b)$ and as an $(a,b)$-geodesic in $G$ when needed. The context will help distinguish to which situation we refer. If $\h$ is a SPG where $S(G,a,b)=\h$ for some graph $G$ and vertices $a,b\in V(G)$, then we say that $G$ is a \textit{base graph} of $\h$. 

In our classification of SPGs, it will be necessary to identify forbidden induced subgraphs of a graph. Given a graph $G$ and $M\subseteq V(G)$, the subgraph of $G$ induced by $M$ is denoted as $G[M]$. 
If $G$ and $H$ are graphs, then $G$ is said to be $H$-free if no induced subgraph of $G$ is isomorphic to $H$. 
A graph with girth $g$, contains a cycle of length $g$, but does not contain a cycle with length smaller than $g$. 

Two pivotal concepts used throughout this paper are index levels
and difference indicies.
Let $G$ be a graph and let $\h = S(G,a,b)$ where the distance from $a$ to $b$ in $G$ is $n+1$. Then, $(a,b)$-geodesics in $G$ have the form $av_1\ldots v_{n}b$, so we will say that $(a,b)$-geodesics in $G$ have $n$ \textit{index levels}, and vertex $v_i$ is at index level $i$ in the $(a,b)$-geodesic. Note that if $v=v_i$, for some $i$ with $1 \leq i \leq n$, then $v$ can only appear in $(a,b)$-geodesics at index level $i$. So, the \textit{index level} of $v$ is defined to be $i$. For the sake of convenience, we will say index level $i$ in a graph $G$ with indicated vertices $a$ and $b$ to be the same as the $i^{\rm th}$ index level of an $(a,b)$-geodesic. Given two $(a,b)$-geodesics $U$ and $W$ in $V(S(G,a,b))$, we define the \textit{difference index set} of $U$ and $W$, denoted as $\diff(U,W)$, as the set of all index levels of the vertices where $U$ and $W$ differ. Note that $U$ and $W$ are adjacent if and only if $|\diff(U,W)| = 1$. If $UW \in E(S(G,a,b))$, the single index level in $\diff(U,W)$ will be called the \textit{difference index} of $UW$.

Much of the work done in this paper revolves around cliques in SPGs. The following theorem provides necessary and sufficient conditions for a SPG to be a clique. 

\begin{theorem}\label{completeGraph}
{\normalfont\cite{AAEHHNW}} $S(G,a,b)=K_n$, for some $n\in \Nnn$,  if and only if each pair of $(a,b)$-geodesics in $G$ differs at the same index. 
\end{theorem}

Extending Theorem~\ref{completeGraph} slightly, we see that for any maximal clique $\calk$ in a SPG, the difference indices of any two edges in $\calk$ are the same. Out of convenience, we say the difference index of $\calk$ is defined to be the common difference index of the edges of $\calk$. Another important observation is that when a $4$-cycle is present in an SPG, the difference indices on the edges of that $4$-cycle must alternate between the same pair of difference indices.

\begin{obs}\label{4cycle_ob}
There are exactly two distinct difference indices among all the edges of any induced $4$-cycle in a SPG and those difference indices must alternate as one traverses the edges of the $4$-cycle. 
\end{obs}

The following proposition from \cite{AAEHHNW} will also be useful.

\begin{proposition}\label{manybasegraphs}
{\normalfont\cite{AAEHHNW}} If $\h=S(G,a,b)$ and $d_G(a,b) = n$, then for any $n'\geq n$ there exists a graph $G'$ with vertices $a, b'\in V(G')$ such that $d_{G'}(a, b') = n'$ and $\h\cong S(G',a,b')$.
\end{proposition}

One of these properties is that SPGs are $C_5$-free. Another---pivotal for characterizing all SPGs with girth $5$ or more---is that if a SPG contains an induced claw, then it must also contain an induced $4$-cycle.

\begin{theorem}\label{noClaw}
{\normalfont\cite{AAEHHNW}} The claw, $K_{1,3}$, is not a SPG. Furthermore, if a SPG $\h$ has an induced claw, then $\h$ has an induced $4$-cycle containing two edges of the induced claw.
\end{theorem}

It was also shown in \cite{AAEHHNW} that if a SPG has an induced $C_k$ for odd $k>5$, then the SPG must contain an induced $4$-cycle. 

\begin{theorem}\label{oddtoC4}
{\normalfont\cite{AAEHHNW}} If a SPG $\h$ has an induced $C_k$ for odd $k > 5$, then $\h$ has an induced $4$-cycle.
\end{theorem}

There were multiple ways in which \cite{AAEHHNW} showed how new SPGs could be created from two existing SPGs. One of these is the disjoint union. That is, $\h_1\cup \h_2$ is a SPG if $\h_1$ and $\h_2$ are SPGs. 

\begin{theorem}\label{disconnect}  
{\normalfont\cite{AAEHHNW}} If  $\h_1$ and $\h_2$  are  SPGs, then $\h_1\cup \h_2$ is a SPG.
\end{theorem}

\section{Structural Results}\label{prelims}

To begin characterizing SPGs with $3$-cycles but no induced $4$-cycles, we first analyze some structures that are forced by requiring no induced $4$-cycles and identify some substructures that are forbidden in SPGs. 
\begin{proposition} \label{3cycle_adj}
Let $\h$ be a SPG containing a $3$-cycle $(U_0,U_1,U_2)$. If $U \in V(\h)\setminus\{U_0,U_1,U_2\}$ with $U \sim U_0$, then either $U$ is adjacent to both $U_1$ and $U_2$ or neither of them.
\end{proposition}

\begin{proof}
By Theorem~\ref{completeGraph}, we can assume that $U_0$, $U_1$, and $U_2$ pairwise differ in a single index level $i$. If $UU_0$ has difference index $i$ then $U\sim U_1$ and $U\sim U_2$. If $UU_0$ has difference index $j$ where $i\neq j$ then $U\not\sim U_1$ and $U\not\sim U_2$.
\end{proof}

There are other restrictions when examining the characteristics of SPGs. 
If $e$ is an edge in a graph $G$, then we denote the graph $G$ with the edge $e$ removed as $G-e$. 

\begin{corollary}\label{k4-e-free}
SPGs are $(K_4-e)$-free. (See Figure~{\normalfont\ref{forbiddenGraphs}}.)
\end{corollary}

We say a subgraph $H$ of $G$ is a maximal clique if $H$ is a clique and the induced subgraph of $V(H)\cup\{v\}$ in $G$ is not a clique for all $v\in V(G)\setminus V(H)$. A quick corollary of Proposition~\ref{3cycle_adj} which is useful for Lemma~\ref{characterCliques} is given below.

\begin{corollary}\label{atMostOne}
Let $\h$ be a SPG and let $\calk$ be a maximal clique in $\h$. Then each vertex in $V(\h)\setminus V(\calk)$ is adjacent to at most one vertex in $V(\calk)$. 
\end{corollary}

The following lemma is pivotal to the constructions in Sections~\ref{girth3}. In particular, this result asserts there are limitations when it comes to the parts of the SPG which have maximal cliques. Let $E(G_1,G_2)$ be the set of edges joining a vertex in $G_1$ with a vertex in $G_2$. Note that depending on the definition of matching used, conditions $(i)$ and $(ii)$ in Lemma~\ref{characterCliques} could be merged together if matchings are allowed to be empty sets. We prefer to think of matchings as being non-empty sets since this allows us to highlight the distinct differences between the conditions listed below.

\begin{lemma}\label{characterCliques}
Let $\h$ be a SPG and let $\calk_1$ and $\calk_2$ be two distinct maximal cliques in $\h$. Then exactly one of the following must be true:
\begin{packedItem}
 \item[$(i)$] no vertex in $\calk_1$ is adjacent to a vertex in $\calk_2$;
 \item[$(ii)$] the set of edges that join vertices in $V(\calk_1)$ to vertices in $V(\calk_2)$ is a matching; or
 \item[$(iii)$] $|V(\calk_1)\cap V(\calk_2)|=1$.
\end{packedItem}
\end{lemma}

\begin{proof}
Suppose that $|V(\calk_1)\cap V(\calk_2)|\geq 2$. Then there is an induced $K_4-e$ in $\h$, contradicting Corollary~\ref{k4-e-free}. We see that $|V(\calk_1)\cap V(\calk_2)|=1$, is precisely case $(iii)$, so assume $|V(\calk_1)\cap V(\calk_2)|=0$. If $|E(\calk_1,\calk_2)|=0$, this falls into case $(i)$. Finally, suppose that $|E(\calk_1,\calk_2)|\geq 1$. By Corollary~\ref{atMostOne}, each vertex in $\calk_1$ is adjacent to at most one vertex in $\calk_2$ and vice versa. Thus there is a matching between $\calk_1$ and $\calk_2$, completing the proof. 
\end{proof}

We can say a bit more about condition $(iii)$ in Lemma~\ref{characterCliques}.

\begin{proposition} \label{adj_cliques}
Let $\calk_1$ and $\calk_2$ be distinct maximal cliques in a SPG that share a single vertex. Then, $\calk_1$ and $\calk_2$ have distinct difference indicies.
\end{proposition}

\begin{proof}
Let $X$ be the vertex shared by $\calk_1$ and $\calk_2$ and let $i$ be the difference index of $\calk_1$. Then, let $W \in V(\calk_2)$. If $\diff(X,W) = \{i\}$, then $W$ is adjacent to all vertices in $V(\calk_1)$, a contradiction.
\end{proof}

\begin{figure}[htb]
\begin{center}
\includegraphics[scale=1]{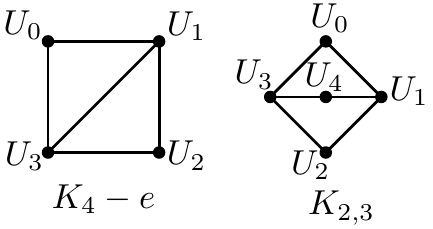}
\vspace{-.1in}
\end{center}
\caption{Forbidden induced graphs in SPGs}\label{forbiddenGraphs}
\end{figure}

The following proposition shows that $K_{2,3}$ is an induced subgraph that is forbidden in SPGs.

\begin{proposition}\label{f3}
SPGs are $K_{2,3}$-free.
\end{proposition}

\begin{proof}
Let $\h$ be a SPG containing an induced subgraph $K_{2,3}$. Label the vertices of the induced $K_{2,3}$ as in Figure~\ref{forbiddenGraphs}. Let the 
difference indices of $U_1U_4$ and $U_3U_4$ be $i$ and $j$, respectively, where $i\neq j$ (by Observation~\ref{4cycle_ob}). Due to the structure of $4$-cycles in $\h$, $U_0U_3$ and $U_0U_1$ have difference indices $i$ and $j$, respectively. Similarly, $U_2U_3$ and $U_1U_2$ must also have difference indices $i$ and $j$, respectively.
But if $U_0U_1$ and $U_1U_2$ both have difference index $j$, then $U_0 \sim U_2$ and so $\h'$ is not isomorphic to $K_{2,3}$, a contradiction.
\end{proof}

The following proposition is not necessary for any of the following results; however, it is an interesting result on its own and follows from Corollary~\ref{k4-e-free} and Proposition~\ref{f3}.
We define the neighborhood of a vertex $v$ in $G$ as $\nbd_G(v)$. 

\begin{proposition}
Let $U$ and $W$ be non-adjacent vertices in a SPG $\h$. Then, exactly one of the following holds:
\begin{itemize}
\item $|\nbd_\h(U)\cap\nbd_\h(W)|\leq 1$; or
\item $|\nbd_\h(U)\cap\nbd_\h(W)|= 2$ and the vertices of $\nbd_\h(U)\cap\nbd_\h(W)$, $U$, and $W$ form an induced $C_4$. 
\end{itemize}
\end{proposition}

\begin{proof}
To get a contradiction, assume there are at least three vertices $Z_1$, $Z_2$, and $Z_3$, contained in $\nbd_\h(U)\cap \nbd_\h(W)$.  By Lemma~\ref{f3}, at least one pair of vertices among $Z_1$, $Z_2$, and $Z_3$ are adjacent.  Without loss of generality, assume $Z_1 \sim Z_2$.  But then $U$, $W$, $Z_1$, and $Z_2$ form an induced $(1,2,2)$-graph, a contradiction of Corollary~\ref{k4-e-free}.

Similarly, if $|\nbd_{\h}(U)\cap \nbd_\h(W)|=2$ and there is an edge between the two vertices in $\nbd_{\h}(U)\cap \nbd_\h(W)$, this forms an induced $(1,2,2)$-graph. Again, this is a contradiction of Corollary~\ref{k4-e-free}.
\end{proof}

\section{SPGs with 3-cycles but no Induced 4-cycles}\label{girth3}

As in \cite{AAEHHNW}, the \textit{one-sum} of two graphs $G$ and $H$ is defined as the graph formed by joining $G$ and $H$ at a single vertex and preserving the edges. 
A graph $\h$ is a \textit{tree of cliques} if 
\begin{packedItem}
\item $\h$ is $C_k$-free for $k\geq 4$, claw-free and 
\item given any two distinct maximal cliques $\h_1$ and $\h_2$ in $\h$, $|V(\h_1)\cap V(\h_2)|\leq 1$. 
\end{packedItem}
Figure~\ref{tree_of_cliques} shows an example of a tree of cliques. Note that trees of cliques can be built by starting with a single clique and repeatedly  attaching maximal cliques using one-sums. 
Also note that a tree of cliques can be disconnected.

\begin{figure}[htb]
\begin{center}
\includegraphics[scale=0.75]{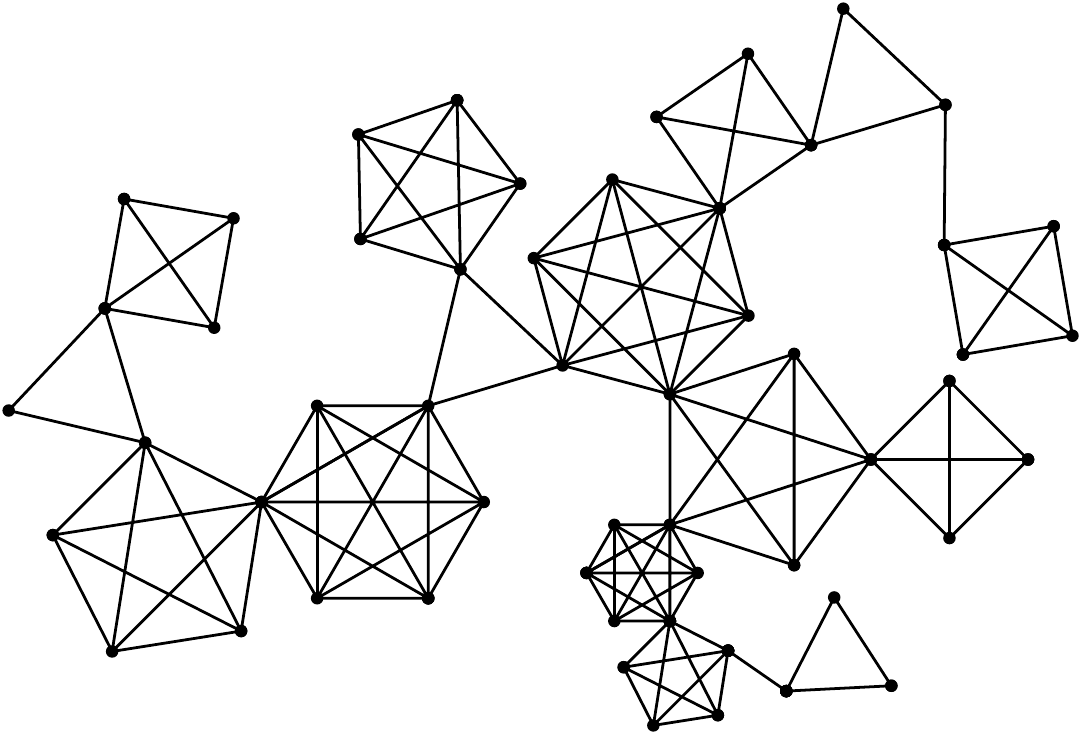}
\end{center}
\caption{A tree of cliques}\label{tree_of_cliques}
\end{figure}

\subsection{Trees of Cliques}\label{Sec:treeOfCliques}

In this section, it is shown that SPGs that are $C_k$-free for $k\geq 4$ are exactly all  trees of cliques. 
The following lemma shows that a clique can be added to a SPG to form another SPG under certain conditions. Note that this is a stronger result than is necessary to characterize SPGs that are $C_k$-free for any $k\geq 4$. Before stating the lemma, we need a definition.
Let $\h = S(G,a,b)$ be a SPG and let $U \in V(\h)$. If there is a vertex $v \in V(G)$ such that $U$ is the unique $(a,b)$-geodesic in $G$ passing though $v$, then $U$ is said to be the \textit{$v$-geodesic}. If it is not necessary to specify the vertex, then we will say that $U$ has the \textit{unique vertex-path property}.

\begin{lemma}\label{Thm:treeClique}
Let $\h = S(G,a,b)$ be a SPG and let $U \in V(\h)$ with the unique vertex-path property. 
Let $\s$ be the one-sum of $\h$ and a maximal clique $\calk$ formed by identifying $U$ with any vertex in $\calk$. 
Then $\s$ is a SPG.
Furthermore, the base graph $G'$ of $\s$ can be constructed from $G$ so that the following two conditions are satisfied:
\begin{packedEnum}
\item [(1)] if $(a,b)$-geodesics in $G$ have at least two index levels,
then 
$(a,b)$-geodesics in $G'$ have the same number of index levels;  
\item[(2)] any vertex in $\h$ other than $U$ that has the unique vertex-path property will still have the unique vertex-path property in $\s$, and every vertex in $\calk$ other than the one identified with $U$ will have the unique vertex-path property in $\s$.
\end{packedEnum}
\end{lemma}

\begin{proof}
Let $U=av_1\cdots v_i\cdots v_pb$, where $U$ is the $v_i$-geodesic. 
If $p=1$, by Proposition~\ref{manybasegraphs}, the length of $(a,b)$-geodesics in $G$ can be extended to any length greater than or equal to $1$ without affecting the SPG, so we may assume $\h$ and $\kcal$ have the same number of index levels. 
Thus, we may assume that $p \geq 2$.
The one-sum of $\h$ and $\calk = K_n$ is the SPG, $S(G',a,b)$, where $G'$ is constructed as follows for $i< p$. Let $V(G') = V(G) \cup \{w_1,\ldots,w_{n-1}\}$ and $E(G') = E(G) \cup \{v_{i}w_j \,|\, j  = 1, \ldots, n-1\} \cup \{w_jv_{i+2} \,|\, j  = 1, \ldots, n-1\}$.  Because $U$ is the unique $(a,b)$-geodesic passing through $v_i$, $G'$ has exactly $n-1$ more $(a,b)$-geodesics than $G$, each of the form $W_j=av_1\cdots v_{i}w_jv_{i+2}\cdots v_pb$, $j = 1,\ldots,n-1$. All of these new $(a,b)$-geodesics are adjacent to one another and to $U$, but are not adjacent to any other $(a,b)$-geodesics in $G$ as desired.

The first condition above is clearly satisfied because if $p \geq 2$, $G'$ is constructed from $G$ by adding vertices at a single existing index level of $G$.
To see the second condition above is satisfied, note that the only vertices in $V(G)$ that are part of new paths in $G'$ are the vertices $\{v_j \, | \,j=1,\ldots,p, j \neq i+1\}$. Because $U$, the $v_i$-geodesic, passes through all the vertices in this set, no other $(a,b)$-geodesic in $\h$ could be the unique $(a,b)$-geodesic passing through one of these vertices.
In addition, recall that the vertices in $V(\calk) \setminus \{U\}$ correspond to $(a,b)$-geodesics of the form $W_j=av_1\cdots v_{i}w_jv_{i+2}\cdots v_pb$, $j = 1,\ldots,n-1$. For each $j = 1,\ldots,n-1$, $W_j$ is the $w_j$-geodesic, and thus the second condition is satisfied. In the case $i = p$, $E(G') = E(G) \cup \{v_{i-2}w_j \,|\, j  = 1, \ldots, n-1\} \cup \{w_jv_{i} \,|\, j  = 1, \ldots, n-1\}$ rather than the set defined above, and the intended SPG is formed by an analogous argument to that above.
\end{proof}

\begin{theorem}\label{Thm:forestOfCliques}
Let $\h$ be a $C_k$-free graph for all $k\geq 4$. Then $\h$ is a SPG if and only if $\h$ is a tree of cliques.
\end{theorem}

\begin{proof}
By assumption, $\h$ is $C_k$-free for $k \geq 4$. Thus, if $\h$ is a SPG, it follows from Theorem~\ref{noClaw} that $\h$ is claw-free. By Lemma~\ref{characterCliques}, $\h$ satisfies the remaining conditions to be a tree of cliques.

For the other direction of the proof, because the disjoint union of SPGs is a SPG by Theorem~\ref{disconnect}, only a single connected SPG need be considered. 
Let $G$ be a base graph defined as follows: $V(G) = \{a,v_1, v_2, \ldots, v_n, w, b\}$ and $E(G) = \{av_i, v_iw | i = 1, \ldots,n\} \cup \{wb\}$. It is easily verifiable that $S(G,a,b) = K_n$. Additionally, each of the $n$ $(a,b)$-geodesics in $G$ is the $v_i$-geodesic for some $i = 1, \ldots, n$, so every vertex in $S(G,a,b)$ satisfies the unique vertex-path property. Using the construction in Lemma~\ref{Thm:treeClique}, it is possible to a create a base graph whose SPG is the one-sum of $S(G,a,b)$ and any clique. In this construction, any vertex in $S(G,a,b)$ that had the unique vertex-path property prior to the construction (other than the vertex being identified in the one-sum) still has the unique vertex-path property in the one-sum. Additionally, each new vertex added to the SPG also has the unique vertex-path property. 
This can be seen in Figure~\ref{tree_of_clique_ex}.
Thus, by repeatedly applying Lemma~\ref{Thm:treeClique}, it is possible to create any such connected SPG. 
\end{proof}

\begin{figure}[htb]
\begin{center}
	\includegraphics[scale=0.75]{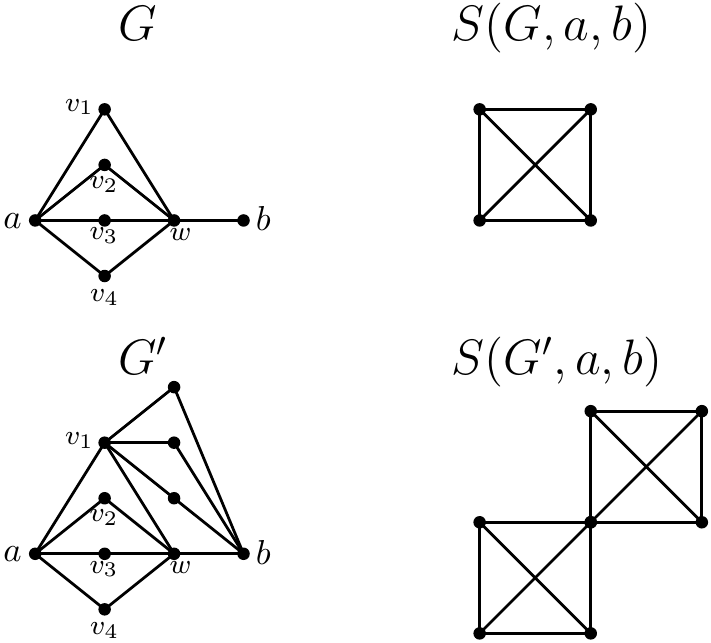}
\end{center}
\caption{Example of identifying a vertex in an SPG with that of a vertex in the clique we are adding to the SPG}\label{tree_of_clique_ex}
\end{figure}

The following observations will be useful when we generalize this result in the next section.

\begin{obs} \label{cliqueParity}
Let $\h = S(G,a,b)$ be a SPG that is a tree of cliques whose base graph $G$ is constructed as in the proof of Theorem~{\normalfont\ref{Thm:forestOfCliques}}. Then, the following properties hold:
\begin{itemize}
\item any vertex in $\h$ that is in exactly one maximal clique has the unique vertex-path property;
\item the base graph of $\h$ can be constructed so that every $(a,b)$-geodesic in $G$ has exactly two index levels;
\item if $\calk_1$ and $\calk_2$ are two maximal cliques in $\h$ that share a single vertex, then the difference index of $\calk_1$ is $i$ and the difference index of $\calk_2$ is $i+1$, or vice versa.
\end{itemize}
\end{obs}

\subsection{Tools To Build New SPGs}\label{Sec:cyclesOfCliques}

Building upon the work in Section~\ref{Sec:treeOfCliques}, we develop tools similar to Lemma~\ref{Thm:treeClique} and Theorem~\ref{Thm:forestOfCliques} 
that will characterize all SPGs that possibly contain a $3$-cycle but no induced $4$-cycle.
Recall that in \cite{AAEHHNW}, it was shown that the only SPGs with girth five or more are even cycles or paths. 
We begin by defining three constructions involving connecting two cliques. 
These constructions will be used to build new SPGs from old SPGs.
Although Constructions~\ref{const:A} and~\ref{const:B} are not used in the proof of the main theorem in Section~\ref{sec:mainthm} (Construction~\ref{const:C} is the only one needed for the proof), 
they provide intuitive descriptions of how to build the base graph of a given SPG. 

\begin{construction}
\label{const:C}
Let $\ucal$ and $\wcal$ be vertex disjoint maximal cliques in a SPG $\h$,
and let $X$ be a vertex that is distinct from the vertices in $\h$. 
Then put an edge between $X$ and all vertices in both $\ucal$ and $\wcal$. 
Construction~\ref{const:C} is shown in Figure~\ref{small_ear_of_cliques3}.
\end{construction}

\begin{figure}[htb]
\begin{center}
 \includegraphics[scale=1]{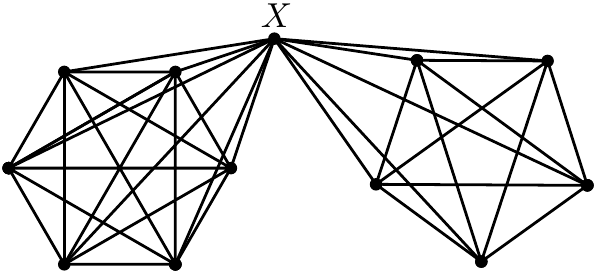}
\end{center}
\caption{Construction~\ref{const:C}}\label{small_ear_of_cliques3}
\end{figure}

\begin{lemma}\label{lem:small_ears_of_cliques}
Let $\h = S(G,a,b)$ be a SPG with exactly two index levels, and let $\ucal$ and $\wcal$ be two disjoint maximal cliques in $\h$ with difference indicies of opposite parity.
Then a SPG $\h'$ can be formed from $\h$ as described in Construction~\ref{const:C} (e.g., Figure~{\normalfont\ref{small_ear_of_cliques3}}).
Furthermore, any vertices in $\h$ that satisfied the unique vertex-path property prior to the construction still satisfy this property in $\h'$.
\end{lemma}

\begin{proof}
Without loss of generality, let $1$ be the index level of $\ucal$ and $2$ be the index level of $\wcal$. Let 
\[
U_i = au_ivb \;\;\; \text{ and } \;\;\; W_j=av'w_jb
\]
be the vertices of $\ucal$ and $\wcal$ for all $i\in\{1,2,\ldots,|\ucal|\}$ and $j\in\{1,2,\ldots,|\wcal|\}$. 
Then we build the base graph $G'$ from $G$ by letting $V(G')=V(G)$ and $E(G')=E(G)\cup \{vv'\}$. 
Under this construction, $S(G',a,b)$ is the graph $\h'$ as described in Construction~\ref{const:C}. Since there are two index levels, it is clear that there are no additional $(a,b)$-geodesics besides $X$. 
The only edges must join $X$ to all vertices in $\ucal$ and $\wcal$ for the same reason.
\end{proof}

\begin{construction}
\label{const:A}
Let $U$ and $W$ be non-adjacent vertices in a SPG $\h$ that are each in exactly one maximal clique. 
Let $\calk_1$ and $\calk_2$ be two disjoint maximal cliques with $|V(\calk_i)| \geq 2$ for $i =1,2$. Let $X$ and $X'$ be distinct vertices in $V(\calk_1)$ and let $Y$ and $Y'$ be distinct vertices in $V(\calk_2)$. Then, build a new graph from $\h$, $\calk_1$, and $\calk_2$ by identifying $X$ with $U$, $Y$ with $W$, and $X'$ with $Y'$. Construction~\ref{const:A} is shown in Figure~\ref{small_ear_of_cliques}. 
\end{construction}

\begin{figure}[htb]
\begin{center}
 \includegraphics[scale=1]{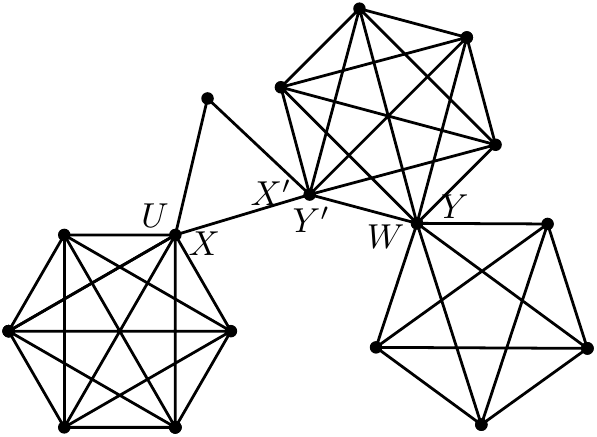}
\end{center}
\caption{Construction~\ref{const:A}}\label{small_ear_of_cliques}
\end{figure}

\begin{construction}
\label{const:B}
Let $U$ and $W$ be non-adjacent vertices in a SPG $\h$ that are each in exactly one maximal clique. Let $\ucal$ and $\wcal$ be vertex disjoint maximal cliques that contain $U$ and $W$, respectively.  
Let $\calk$ be a clique with $|V(\calk)| \geq 2$. Let $X$ and $X'$ be distinct vertices in $V(\calk)$. Then, build a new graph from $\h$ and $\calk$ by identifying $X$ with $U$ and putting an edge between $X'$ and every vertex in $\wcal$. Alternatively, we could build a different graph from $\h$ and $\calk$ by identifying $X$ with $W$ and put an edge between $X'$ and every vertex in $\ucal$. Construction~\ref{const:B} is shown in Figure~\ref{small_ear_of_cliques2}.
\end{construction}

\begin{figure}[htb]
\begin{center}
 \includegraphics[scale=1]{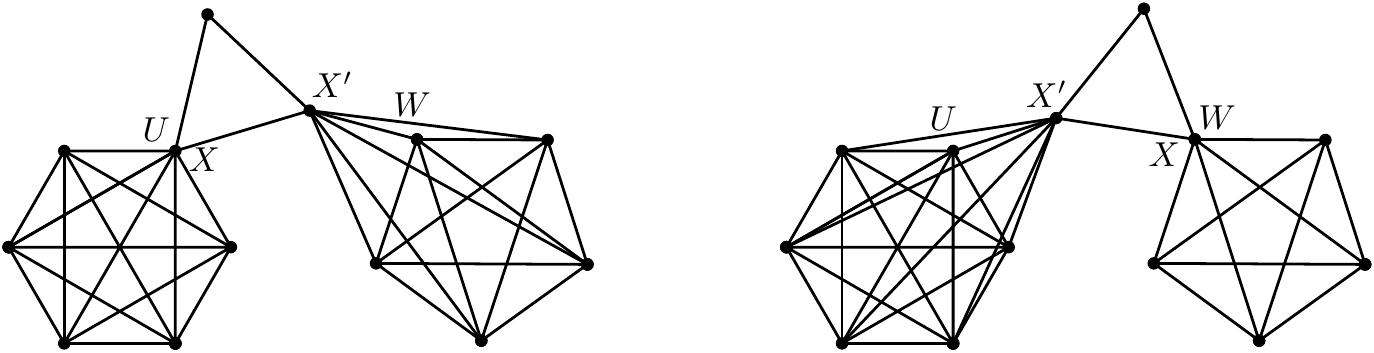}
\end{center}
\caption{Construction~\ref{const:B}}\label{small_ear_of_cliques2}
\end{figure}

\begin{lemma}\label{small_ears_of_cliques}
Let $\h = S(G,a,b)$ be a SPG and let $U$ and $W$ be non-adjacent vertices in $V(\h)$.
Suppose that $U$ is the $u$-geodesic and $W$ is the $w$-geodesic for vertices $u,w \in V(G)$. If $|\diff(U,W)| \leq 2$ 
and $i_1$ and $i_2$ are the indicies of $u$ and $w$, respectively, then the graph $\h'$ formed from $\h$ by joining $U$ and $W$ to cliques in the following ways is an SPG:

\begin{packedEnum} 
\item[(1)] If $i_1 \neq i_2$, then $\h'$ is the graph defined in Construction~\ref{const:A} (e.g., Figure~{\normalfont\ref{small_ear_of_cliques}}).
\item[(2)] If $i_1 = i_2$, then $\h'$ is one of the graphs defined in Construction~\ref{const:B} (e.g., Figure~{\normalfont\ref{small_ear_of_cliques2}}).
\end{packedEnum}

Furthermore, any vertices in $V(\h)\setminus\{U,W\}$ (Construction~\ref{const:A} only), or either $V(\h)\setminus\{U\}$ or $V(\h)\setminus\{W\}$ (Construction~\ref{const:B} only), and vertices in the new cliques not including $X,X',Y,Y'$ that satisfied the unique vertex-path property prior to the use of Construction~\ref{const:A} or~\ref{const:B} will still satisfy this property in $\h'$.
\end{lemma}

\begin{proof}
Note that since $U$ and $W$ are non-adjacent, $|\diff(U,W)|=2$.
Since the disjoint union of SPGs is a SPG by Lemma~\ref{disconnect}, let $\h$ be connected. From Observation~\ref{cliqueParity}, it follows that $U$ and $W$ are in exactly one maximal clique as shown in Figures~\ref{small_ear_of_cliques3}, ~\ref{small_ear_of_cliques}, and~\ref{small_ear_of_cliques2}. 
Since there are no $4$-cycles in $\h$ or $\h'$ and $|\diff(U,W)|=2$, $|i_1-i_2|\leq 1$ and if $i_1=i_2$, then the other difference index in $\diff(U,W)$ must be $i_1+1$ or $i_1-1$.
Notice there are no additional $(a,b)$-geodesics in $G$ passing through $u$ and $w$ as a result. 
Without loss of generality, assume that $i_1 \leq i_2$ and let $U$ and $W$ be the following:

\[
U=av_1\cdots v_{i_1-1}v_{i_1}v_{i_1+1}v_{i_1+2}\cdots v_pb \;\;\text{ and }\;\; W=av_1\cdots v_{i_1-1}w_{i_1}w_{i_1+1}v_{i_1+2}\cdots v_pb \text{.}
\]  
With this notation, note that $u = v_{i_1}$. If $i_1 < i_2 = i_1 +1$, then  $w =w_{i_1 +1}$. If $i_1 = i_2$ then $w = w_{i_1}$. Also, note that $v_{i} \neq w_{i}$ for $i \in \{i_1,i_1 +1\}$.   

First, assume $i_1 < i_2$. 
To perform Construction~\ref{const:A}, we build the base graph $G_1$ from $G$ as follows. The vertex set $V(G_1)=V(G)\cup \{u_1,\ldots, u_{k_1-2},u_1',\ldots,u_{k_2-2}'\}$, and the edge set 
\begin{align*}
E(G_1)&=E(G)\cup \{v_{i_1}w_{i_1+1}\}\cup\{v_{i_1-1}u_i, u_iw_{i_1+1} \,:\, i\in \{1,\ldots,k_1-2\}\}\\
&\;\;\;\;\;\;\;\;\;\;\;\;\;\cup \{v_{i_1}u_i', u_i'v_{i_1+2}\,:\, i\in\{1,\ldots,k_2-2\}\}\text{,}
\end{align*}
where $k_1$ and $k_2$ are the numbers of vertices in the new cliques added to $\h$ containing $W$ and $U$, respectively (as shown in Figure~\ref{small_ear_of_cliques}). Under this construction, $S(G_1,a,b)$ is the graph $\h'$ with $U$ and $W$ connected as described in Construction~\ref{const:A}.

To see that all vertices in $\h$ (other than $U$ or $W$) that satisfied the unique vertex-path property prior to the construction still satisfy this property after the construction, note that the only vertices in $V(G)$ that are a part of new paths in $G_1$ are in the following set: $\{v_1, \ldots, v_{i_1}, w_{i_1+1}, v_{i_1+2}, \ldots, v_p\}$. Because either $U$ or $W$ passes through each of these vertices, no $(a,b)$-geodesic in $G$ other than $U$ or $W$ can be the unique $(a,b)$-geodesic passing through any of these vertices. Moreover, every new vertex added to $\h$, other than $X'$ (or $Y'$), is either the $u_i$-geodesic for some $i \in \{1, \ldots, k_1 - 2\}$ or the $u_i'$-geodesic for some $i \in \{1, \ldots, k_2 - 2\}$.

Now, assume $i_1 = i_2$. 
To perform Construction~\ref{const:B}, we form the base graph $G_2$ with the vertex set $V(G_2)=V(G)\cup \{u_1,\ldots, u_{k_1-2}\}$ and edge set 
\[
E(G_2)=E(G)\cup \{v_{i_1}w_{i_1+1}\}\cup\{v_{i_1}u_i, u_iv_{i_1+2} \,:\, i\in \{1,\ldots,k_1-2\}\}\text{,}
\]
where $k_1$ is the number of vertices in the new clique in $\h$ containing $U$, as shown in the left graph of Figure~\ref{small_ear_of_cliques2}. Under this construction, $S(G_2,a,b)$ is the graph $\h$ with $U$ and $W$ connected as shown in the first graph of Construction~\ref{const:B}. To form the second graph in Construction~\ref{const:B}, the edge set
\[
E(G_2)=E(G)\cup\{w_{i_1},v_{i_1+1}\}\cup\{w_{i_1}u_i,u_iv_{i_1+2}\,:\, i\in \{1,\ldots,k_1-2\}\}
\]
is used instead of the edge set described above.

If $i_1 = i_2$, it also possible that $W$ has the form, $W=av_1\cdots v_{i_1-2}w_{i_1-1}w_{i_1}v_{i_1+1}\cdots v_p b$. If this is the case, the graph shown in Construction~\ref{const:B} can be constructed using an argument analogous to the one above.
\end{proof}

\subsection{Classification of $C_4$-free SPGs}\label{sec:mainthm}

Putting this all together, in this section we characterize all SPGs with no induced $4$-cycles. 
We define a \textit{collection of cliques} as a graph that is claw-free and given any two distinct maximal cliques $\h_1$ and $\h_2$ in $\h$, $|V(\h_1)\cap V(\h_2)|\leq 1$. That is, a collection of cliques is a tree of cliques that can contain cycles. 
An example of a collection of cliques is shown in Figure~\ref{cyclesOfCliques}. Note that a collection of cliques may be disconnected. 

\begin{figure}[htb]
\begin{center}
\includegraphics[scale=0.75]{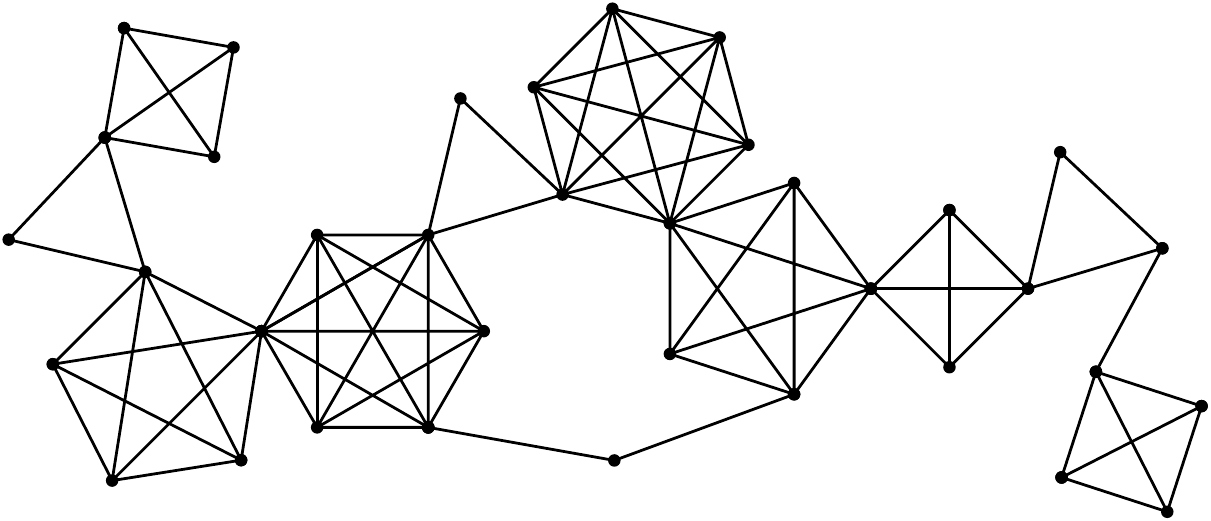}
\end{center}
\caption{A collection of cliques}\label{cyclesOfCliques}
\end{figure}

\begin{theorem}\label{no4cycles_classification}
Let $\h$ be a  $C_4$-free graph. Then $\h$ is a SPG if and only if $\h$ is a collection of cliques that is $C_k$-free for all odd $k \geq 5$. Furthermore, the base graph of $\h$ can be constructed with only two index levels. 
\end{theorem}

\begin{proof}
By assumption, $\h$ is $C_4$-free. Thus, if $\h$ is a SPG, it follows from Theorem~\ref{noClaw} that $\h$ is claw-free. 
It was shown in \cite{AAEHHNW} that SPGs are $C_5$-free. From Theorem~\ref{oddtoC4}, it follows that $\h$ is $C_k$-free for any odd $k > 5$. 
By Lemma~\ref{characterCliques}, $\h$ satisfies the remaining condition to be a collection of cliques. 

Now, assume that $\h$ is a collection of cliques that is $C_k$-free for all odd $k \geq 5$. 
By Theorem~\ref{disconnect}, we can assume that $\h$ is connected.
Any such graph that does not have an induced $C_k$ for $k >3$ is a tree of cliques and thus any such $\h$ is an SPG by Theorem~\ref{Thm:forestOfCliques}.
We can now assume that $\h$ contains an induced cycle of even length $k >4$. 
We proceed by induction on the number of vertices in $\h$.
Let $|V(\h)|=n$ and assume any collection of cliques that is $C_4$-free, $C_k$-free for odd $k \geq 5$, and with fewer than $n$ vertices is a SPG, and the base graph of $\h$ can be constructed with two index levels.

Let $X$ be a vertex in  $V(\h)$ shared between two distinct maximal cliques, $\h_1$ and $\h_2$, that are part of an induced even cycle contained in $\h$. Let $\h'$ be the graph formed from $\h$ by deleting $X$. 
By the induction hypothesis, $\h' = S(G,a,b)$ where $G$ only has two index levels. 
Because $X$ was the only vertex in $V(\h)$ joining $\h_1$ and $\h_2$, in $\h'$, $\h_1-X$ and $\h_2-X$ are disjoint. Because $\h_1$ and $\h_2$ are part of an induced cycle with even length in $\h$, $\h'$ is connected, and from Proposition~\ref{adj_cliques}, it follows that the difference indicies of $\h_1-X$ and $\h_2-X$ have opposite parity. 
By Lemma~\ref{lem:small_ears_of_cliques}, we can form a base graph $G'$ such that $S(G',a,b)=\h$ and $G'$ has two index levels as desired.
By induction, the result follows.
\end{proof}

\section*{Acknowledgments}

The authors are grateful to AIM and to the organizers of the REUF program at AIM for making this
collaboration possible. This project was initiated as part of the REUF program at AIM, NSF grant DMS 1620073. We would also like to thank Beth Novick and Ruth Haas for their helpful suggestions in preparing this paper. 

\bibliographystyle{amsplain}
\bibliography{vdec}

\end{document}